\documentclass[10pt, a4paper]{article}
\usepackage[english]{babel}
\usepackage{amsmath}
\usepackage{cmll}
\usepackage{amsthm}
\usepackage{amssymb}
\usepackage{bbm}
\usepackage{mathrsfs}
\usepackage{stmaryrd}
\usepackage[all]{xy}
\usepackage{bussproofs}

\newcommand{\sse}{\leftrightarrow}
\newcommand{\arr}{\longrightarrow}

\newcommand{\imply}{\to}
\newcommand{\C}{\mathbb{C}}

\newcommand{\pp}{\textbf{ISL}}
\newcommand{\ed}{elementary existential}

\newdir{|>}{%
!/4.5pt/@{|}*:(1,-.2)@^{>}*:(1,+.2)@_{>}}

\mathcode`\:="603A     % : = \colon
\mathchardef\colon="303A  % :=

\mathcode`\<="4268     % < = \langle
\mathcode`\>="5269     % > = \rangle
\mathchardef\gt="313E  % arithmetic
\mathchardef\lt="313C  % strict order

\theoremstyle{definition}
\newtheorem{deff}{Definition}[section]
\newtheorem{prop}[deff]{Proposition}
\newtheorem{exa}[deff]{Example}
\newtheorem{lem}[deff]{Lemma}
\newtheorem{cor}[deff]{Corollary}
\newtheorem{Remark}[deff]{Remark}

\begin{document}
\title{A sheafification theorem for doctrines}
\author{Fabio Pasquali\footnote{Part of the research in this paper was carried out while the author worked at Utrecht University in the NWO-Project `The Model Theory of Constructive Proofs' nr.  613.001.007.}}
\maketitle

\begin{abstract} We define the notion of sheaf in the context of doctrines. We prove the associate sheaf functor theorem. We show that grothendieck toposes and toposes obtained by the tripos to topos construction are instances of categories of sheaves for a suitable doctrine.  
\end{abstract}
\section*{Introduction}\label{sec0}
In this paper we study the notions of sheaf in the context of doctrines. A doctrine, whose definition was introduced by Lawvere \cite{LawAdj, LawDiag,LawEq}, can be informally though of as a category $\C$ with a chosen internal logic $P$. In this context we define an object of $\C$ to be a sheaf if it orthogonal to a class of morphisms of $\C$, which satisfy appropriate conditions formulated in the logic of $P$. We shall refer at those objects as $P$-sheaves and by we denote by \textbf{Shv}$(\C,P)$ the full subcategory of $\C$ on $P$-sheaves.\\\\
For a suitable class of doctrines we give a proof of the associate sheaf functor theorem, i.e. of the fact that \textbf{Shv}$(\C,P)$ is a reflective subcategory of $\C$. The proof is entirely internal in the logic of the doctrine and it is a generalization of the one given in \cite{veit} and sketched in \cite{Jacobs}, which is a partially internal proof, carried out in the internal logic of an arbitrary elementary topos.\\\\
As an application, we show that grothendieck toposes and toposes obtained by the tripos to topos construction (such as the effective topos \cite{tripinret}) are categories of the form \textbf{Shv}$(\C,P)$ for a suitable doctrine.\\\\
Sections \ref{sec1} and \ref{sec7} introduces doctrines and sheaves respectively. In section \ref{sec17} we prove the associate sheaf functor theorem for an appropriate class of doctrine doctrines. In section \ref{sec1a} we give the definition of complete objects, which turns out to be useful in discussing some application in section \ref{sec7b}.\\\\
\textbf{Acknowledgments}. The author is grateful to Giuseppe Rosolini, Jaap van Oosten and Ruggero Pagnan for their indispensable remarks.
\section{Preliminary definitions on doctrines}\label{sec1}
In this section we recall some definitions concerning doctrines. Our notation closely follows the one in \cite{RM2}.\\\\Let \textbf{ISL} be the category of inf-semilattices and homomorphisms between them.
\begin{deff}\label{def1}
A doctrine is a pair $(\C, P)$ where $\C$ is a category with finite products and $P$ a functor$$P:\C^{op}\arr \pp$$
\end{deff}
We shall refer to $\C$ as the base category of the doctrine. We will write $f^*$ instead of $P(f)$ to denote the action of the functor $P$ on the morphism $f$ of $\C$, we will often call it reindexing along $f$. Binary meets in inf-semilattices are denoted by $\wedge$. Elements in $P(A)$ will often be called formulas over $A$ and the top element is denoted by $\top_A$.
\begin{exa}
If $\C$ is a category with finite limits, we shall denote by $(\C,\text{sub})$ the doctrine of subobjects of $\C$. The functor sub maps every object $A$ of $\C$ to sub$(A)$, the collection of subobjects over $A$. Top elements are represented by identities, while binary meets and reindexing are provided by pullbacks.   
\end{exa}
\begin{deff}\label{def2}
A doctrine $(\C,P)$ is \textbf{elementary} \textbf{existential} if for each arrow $f : X \arr Y$ in $\C$ there exists a functor $\exists_{f} \dashv f^*$ satisfying
\begin{itemize}
\item[-]Beck-Chevalley condition: i.e. for every pullback of the form 
\[\xymatrix{
X\ar[d]_-{g}\ar[r]^-{f}&Y\ar[d]^{h}\\
Z\ar[r]_-{k}&W
}\]
it holds that $\exists_f\circ g^*= h^*\circ\exists_k$
\item[-]Frobenius Reciprocity: i.e. $\exists_f (\alpha\wedge f^*\beta) = \exists_f \alpha \wedge \beta$.
\end{itemize}
\end{deff}
\begin{Remark}\label{ginapina}For every object $A$ in $\C$ the formula $\exists_{\Delta_A}\top_A$ is in $P(A\times A)$. We will abbreviate it by $\delta_A$ and we shall refer at it as the \textbf{equality predicate} over $A$. The formula $\delta_A$ is \textbf{substitutive}, i.e. for every $X$ in $\C$ and every $\phi$ in $P(X\times A)$ it holds that
$$<\pi_1,\pi_2>^*\phi\wedge<\pi_2,\pi_3>^*\delta_A\le<\pi_1,\pi_3>^*\phi$$
where $\pi_1$, $\pi_2$ and $\pi_3$ are projections from $X\times A\times A$. Moreover for a morphism $f:A\arr B$ in $\C$ and $\alpha$ a formula in $P(A)$ we have $$\exists_f\alpha = \exists_{\pi_1}((id_B\times f)^*\delta_B\wedge\pi_2^*\phi)$$
where $\pi_1$ and $\pi_2$ are projections from $B\times A$ \cite{LawEq}.\end{Remark}
Given a doctrine $(\C,P)$, an object $A$ of $\C$ and a formula $\alpha$ in $P(A)$, we say that the doctrine has a \textbf{comprehension} of $\alpha$ if there exists a morphism $\lfloor\alpha\rfloor:X\arr A$ such that $\top_X \le \lfloor\alpha\rfloor^*\alpha$ and for every $f:Y\arr A$ such that $\top_Y\le f^*\alpha$ there exists a unique $h:Y\arr X$ with $\lfloor\alpha\rfloor \circ h = f$.
%Moreover the comprehension is said to be \textbf{full} if for every $\beta$ in $P(A)$ it holds that $\lfloor\alpha\rfloor^*\alpha\le\lfloor\alpha\rfloor\beta$ if and only if $\alpha\le\beta$.\\\\
%We shall say that a doctrine $(\C,P)$ has (full) comprehensions if for every object $A$ and every formula $\alpha$ in $P(A)$ there exists a (full) comprehension of $\alpha$.
\begin{deff}
In a \ed\ doctrine $(\C,P)$ a morphism $f$ is said to have an \textbf{image} if there exists a comprehension of $\exists_f\top_X$.
\end{deff}
\begin{deff}\label{nonnonanni}
A doctrine $(\C,P)$ has \textbf{power objects} if for every $X$ in $\C$, there exists $\mathbb{P}(X)$ in $\C$ and a formula $\in_X$ in $P(X\times \mathbb{P}(X))$ such that for every object $Y$ in $\C$ and formula $\gamma$ in $P(X\times Y)$ there exists a unique morphism $\{\gamma\}:Y\arr \mathbb{P}(X)$ such that $\gamma = (id_X\times \{\gamma\})^*\in_X$.
\end{deff}
\begin{Remark}\label{walter}If a doctrine $(\C,P)$ has power objects, then for every morphism $f:A\arr\mathbb{P}(B)$ we have that $f = \{(id_B\times f)^*\in_B\}$.
\end{Remark}
\begin{exa}
The subjects doctrine of an elementary topos $\mathcal{E}$ is an \ed\ doctrine with power objects \cite{Jacobs}. Also localic triposes, i.e. triposes of the form $(\text{Sets},\mathbb{H}^{(-)})$, for $\mathbb{H}$ a locale, are \ed\ doctrines with power objects \cite{Tripinret}. But the realizability triposes are not, since arrows of the form $\{\gamma\}$ need not to be unique \cite{Jacobs}.
\end{exa}
Some of the definitions and some the proofs that follows in the next sections are more readable if written in the \textbf{internal language} and we introduce it recalling the definition of Pitts in \cite{Tripinret}:
let $(\C,P)$ be an \ed\ doctrine; let $\Sigma_P$ be the signature which has a  sort for each object of $\C$, an $n$-ary function symbols for each morphism in $\C$ of the form $X_1\times X_2 \times . . . X_n\arr X$ and an $n$-ary relation symbols for each element of $P(X_1\times X_2 \times . . . X_n)$; the internal language of $(\C,P)$ is made by those terms and formulas over $\Sigma_P$. Thus, for an \ed\ doctrine, formulas in the internal language are written in the regular fragment of logic (i.e $\wedge$, $\top$, $=$ and $\exists$).

\section{Complete objects}\label{sec1a}
In an \ed\ doctrine $(\C,P)$ a formula $F$ in $P(Y\times A)$ is a functional relation from $Y$ to $A$ if $$F(y,a) \wedge F(y,a') \le \delta_A(a,a')$$ $$\top_Y(y) = \exists a:Y.\ F(y,a)$$
We shall denote by \textbf{Map}$(\C,P)$ the category whose objects are those of $\C$ and whose morphisms are functional relations of $(\C,P)$: identities are internal equalities and the composition is the usual composition of relations, i.e. if $F$ is in $P(A\times B)$ and $G$ in $P(B\times C)$, the composition $G\circ F$ is the formula of $P(A\times C)$ determined by $$\exists b:B.\ (F (a,b)\wedge G(b,c))$$
There exists a functor $\Gamma:\C\arr \textbf{Map}(\C,P)$ which is the identity on objects an maps a morphism $f:Y\arr A$ to the formula $\Gamma f = (f\times id_A)^*\delta_A$ in $P(Y\times A)$ which is a functional relation from $Y$ to $A$. $\Gamma f$ is said the \textbf{internal graph} of $f$. 
\begin{deff}An object $A$ of $\C$ is said to be \textbf{complete} if for every $Y$ in $\C$ and for every functional relation $F$ from $Y$ to $A$ there exists a unique morphism $f:Y\arr A$ such that $$\Gamma f = F$$
\end{deff}
We shall denote by \textbf{Map}$_c(\C,P)$ the full subcategory of \textbf{Map}$(\C,P)$ on complete objects. There exists a functor $$\nabla:\textbf{Map}_c(\C,P)\arr\C$$ which is the identity on objects and maps every functional relation $F$ to the unique morphism of $\C$ whose internal graph is $F$, which fits in the following commutative diagram
\[
\xymatrix{
\textbf{Map}_c(\C,P)\ar@{^{(}->}[r]\ar[dr]_-{\nabla}&\textbf{Map}(\C,P)\\
&\C\ar[u]_-{\Gamma}
}
\]
\begin{Remark}\label{carcomplete}
(Rosolini \cite{ccR}) Given an \ed\ doctrine $(\C,P)$, we say that a formula $L$ in $P(Y\times A)$ is a left adjoint if there exists a formula $R$ in $P(A\times Y)$ such that $$\delta_Y(y,y') \le \exists a:A.\ (L(y,a) \wedge R(a,y'))$$ $$\exists y:Y.\ (R(a,y) \wedge L(y,a'))\le \delta_A(a,a')$$
Thus an object $A$ is complete if and only if for every $Y$ and every left adjoint formula $L$ in $P(Y\times A)$ there exists a unique $f:Y\arr A$ such that $\Gamma f = L$.\\ In fact, given a compete object $A$ and formulas $L$ and $R$ as above $L(y,a) \wedge R(a,y)$ is functional from $Y$ to $A$, then, since $A$ is complete, there exists a unique morphism $f:Y\arr A$ such that$$\delta_A(f(y),a) = L(y,a) \wedge R(a,y)$$ from which $\delta_A(f(y),a) \le L(y,a)$. Moreover $\top_{Y\times Y} = R(f(y),y)$, hence $$L(y,a) \le R(f(y),y) \wedge L(y,a) \le \exists z:Y. R(f(y),z) \wedge L(z,a) \le \delta_A(f(y),a)$$ and therefore $L(y,a) = \delta_B(f(y),a)$. Conversely, every functional relation $F$ is left adjoint to $F^{op}=<\pi_2,\pi_1>^*F$. 
\end{Remark}

\section{Sheaves in doctrines}\label{sec7}
\begin{deff}\label{dense}
In an \ed\ doctrine $(\C,P)$ a morphism $f:A\arr B$ of $\C$ is said to be \textbf{internally bijective} whenever
$$\delta_A=(f\times f)^*\delta_B$$ $$\top_B = \exists_f\top_A$$
\end{deff}
In the internal language the first condition is $$\delta_A(a,a')=\delta_B(f(a),f(a'))$$ and when it holds we shall say that $f$ is \textbf{internally injective}. On the other hand the second condition is $$\top_B(b) = \exists a:A.\ \delta_B(f(a),b)$$ and when it holds we shall say that $f$ is \textbf{internally surjective}.
\begin{exa}Given an elementary topos $\mathcal{E}$ and its doctrine of subobjects $(\mathcal{E},\text{sub})$, internally injective morphisms are the monomorphisms and internally surjective morphisms are the epimorphisms. Therefore the class of internally bijective arrows is the class of isomorphisms \cite{TT}.\end{exa} 
\begin{deff}\label{sheaf}
In an \ed\ doctrine $(\C,P)$ an object $A$ of $\C$ is a $P$-\textbf{sheaf} if for every span in $\C$ of the form
\[\xymatrix{Y&X\ar[l]_-{d}\ar[r]^-{q}&A}\]
with $d$ internally bijective, there exists a unique $h:Y\arr A$ with $h\circ d = q$.
\end{deff}
In other words, $A$ is a $P$-sheaf if it is orthogonal to the class of internally bijective morphisms of $(\C,P)$. We shall denote by \textbf{Shv}$(\C,P)$ the full subcategory of $\C$ on $P$-sheaves.
\begin{prop}\label{ccu}
In every \ed\ doctrine $(\C,P)$ if an object is complete, then it is a $P$-sheaf.
\end{prop}
\begin{proof}
Suppose $A$ is complete and take the following span in $\C$ 
\[
\xymatrix{Y&X\ar[l]_-{d}\ar[r]^-{q}&A}
\]
with $d$ internally bijective and define $F$ in $P(Y\times A)$ as $$F(y,a) = \exists x:X.\ \delta_A(q(x),a) \wedge \delta_Y(d(x),y)$$
by Frobenius Reciprocity we have that $$F(y,a)\wedge F(y,a') \le \exists x\ x':X.\  \delta_A(q(x),a)\wedge \delta_Y(d(x),d(x'))\wedge\delta_A(q(x'),a')$$  which, by internal injectivity of $d$ and repetitive use of substitutivity of $\delta$, brings to $F(y,a)\wedge F(y,a')\le \delta_A(a,a')$. Moreover
\begin{equation}\notag
\begin{split}
\exists a:A.\ F(y,a) & = \exists x:X.\ \exists a:A.\ \delta_A(q(x),a) \wedge \delta_Y(d(x),y) \\
& = \exists x:X.\ \delta_Y(d(x),y) = \top_Y
\end{split}
\end{equation}
again using Frobenius reciprocity and internal surjectivity of $d$.\\Then, since $A$ is complete, there exists a unique morphism $h:Y\arr A$ such that $F(y,a) = \delta_A(h(y),a)$.\\\\Because of internal injectivity of $d$ we have $$\delta_A(h(d(x)),a) = \exists x':X.\  \delta_A(q(x'),a) \wedge \delta_X(x',x) = \delta_A(q(x),a)$$ thus $q$ and $h\circ d$ have the same internal graph, therefore they are equal by completeness of $A$.
\end{proof}
As a corollary of \ref{ccu}, we have that there exists a functor $U:\textbf{Map}_c(\C,P)\arr \textbf{Shv}(\C,P)$ which is the identity on objects and maps every functional relation $F$ to the unique morphism $f$ such that $\Gamma f = F$. Thus the functor $\nabla:\textbf{Map}_c(\C,P)\arr\C$ factors through the inclusion of \textbf{Shv}$(\C,P)$ in $\C$ as in the following commutative diagram
\[
\xymatrix{
\textbf{Map}_c(\C,P)\ar@{^{(}->}[r]\ar[rd]^-{\nabla}\ar[d]_-{U}&\textbf{Map}(\C,P)\\
\textbf{Shv}(\C,P)\ar@{^{(}->}[r]&\C\ar[u]_-{\Gamma}
}
\]
An \ed\ doctrine $(\C,P)$ is said to \textbf{admit sheafification}  if $P$-\textbf{Shv} is reflective.
In the next section we will introduce a class of doctrines which admit sheafification. For those doctrines it holds also that both the functor $U$ and the inclusion of \textbf{Map}$_c(\C,P)$ into \textbf{Map}$(\C,P)$ are equivalences. This fact turns out to be useful to prove in section \ref{sec7b} that every topos that comes from a tripos via the tripos to topos construction is a category of the form \textbf{Shv}$(\C,P)$ for some suitable doctrine $(\C,P)$.
\section{A sheafification theorem}\label{sec17}
\begin{deff}\label{mizzega} 
An \ed\ doctrine $(\C,P)$ is said to have \textbf{singletons} if\\\\
i) $(\C,P)$ has power objects\\\\
ii) arrows of the form $\{\delta_A\}$ have images and are internally injective\\\\iii) for every morphism $g:Y\arr \mathbb{P}(A)$ it holds that $(id_A\times g)^*\in_A$ is a functional relation from $Y$ to $A$ if and only if $g^*(\exists_{\{\delta_A\}}\top_A) = \top_Y$. 
\end{deff}
\begin{exa}
The doctrine of subobjects of an elementary topos $(\mathcal{E}, \text{sub})$ has singletons. In fact for every $A$ the arrow $\delta_A$ is represented by the diagonal $\Delta_A:A\arr A\times A$ and $\{\delta_A\}$ is the unique arrow that makes the following a pullback
\[
\xymatrix{
A\ar[d]_-{\Delta_A}\ar[rr]&&.\ar[d]^-{\in_A}\\
A\times A\ar[rr]_-{id_A\times\{\delta_A\}}&&A\times \mathbb{P}(A)
}
\]
$\{\delta_A\}$ is mono and therefore internally injective in $(\mathcal{E},\text{sub})$. The doctrine has all images, given by the epi-mono factorization of $\mathcal{E}$. Given a morphism $g:Y\arr\mathbb{P}(A)$, the condition $g^*(\exists_{\{\delta_A\}}\top_A)=\exists_{\pi}(\{\delta_A\}\times g)^*\delta_{\mathbb{P}(A)} = \top_Y$ can be written internally as $$\exists a: A.\ \forall x:A.\ (\delta_A(a,x) \sse x\in_Ag(y))$$ this proposition can be easily seen to be equivalent to the following $$\exists a: A.\ (a\in_A g(y) \wedge \forall x:A.\ x\in_Ag(y)\imply\delta_A(a,x))$$ which expresses the fact that $(id_A\times g)^*\in_A$ is a functional relation \cite{TT}.
\end{exa}
The following three lemmas are instrumental to prove in \ref{trip} that the condition of having singletons is sufficient for an \ed\ doctrine to admit sheafification. Thus in the rest of the section we will assume to work with an \ed\ doctrine $(\C,P)$ with singletons, we shall abbreviate the formula $\exists_{\{\delta_A\}}\top_A$ with $\sigma_A$ and denote by $S_A$ the image of $\{\delta_A\}$, i.e. the domain of the arrow $$\lfloor\sigma_A\rfloor:S_A\arr \mathbb{P}(A)$$
\begin{lem}\label{trip1}
For every object $A$ in $\C$, there exists a morphism $\eta_A:A\arr S_A$ such that $\eta_A$ is internally bijective.
\end{lem}
\begin{proof}
We trivially have that $\top_A = \{\delta_A\}^*\sigma_A$, therefore, by the universal property of comprehensions, there exists $\eta_A:A\arr S_A$ with $\lfloor\sigma_A\rfloor\circ\eta_A=\{\delta_A\} $. Internally injectivity of $\eta_A$ follows from internal injectivity of both $\{\delta_A\}$ and $\lfloor\sigma_A\rfloor$. Moreover, by definition of comprehension, we have that 
\begin{equation}\notag
\begin{split}
\top_{S_A}(s)& = \exists a: A.\ \delta_{\mathbb{P}(A)}(\{\delta_A\}(a), \lfloor\sigma_A\rfloor(s))\\
& = \exists a: A.\ \delta_{\mathbb{P}(A)}(\lfloor\sigma_A\rfloor(\eta_A(a)), \lfloor\sigma_A\rfloor(s))\\
& = \exists a: A.\ \delta_{S_A}(\eta_A(a),s)
\end{split}
\end{equation}
which proves the internal surjectivity of $\eta_A$.
\end{proof}
\begin{lem}\label{trip2}
For every $A$ in $\C$ it holds that $\delta_{S_A}(\eta_A(a),s) = a\in_A\lfloor\sigma_A\rfloor(s)$.
\end{lem}
\begin{proof}
Since $\{\delta_A\} = \lfloor\sigma_A\rfloor \circ \eta_A$, we have $$\delta_A=(id_A\times\{\delta_A\})^*\in_A=(id_A\times \eta_A)^*(id_A\times \lfloor\sigma_A\rfloor)^*\in_A$$ therefore $$\exists_{id_A\times \eta_A}\delta_A \le (id_A\times \lfloor\sigma_A\rfloor)^*\in_A$$This inequality can be written internally as
$$\exists x,y:A.\ \delta_A(x,a)\wedge \delta_{S_A}(\eta_A(x),s)\wedge \delta_A(x,y)\le a\in_A\lfloor\sigma_A\rfloor(s)$$
by Frobenius reciprocity and substitutivity of $\delta_A$, the left hand side of the inequality is equal to $\delta_{S_A}(\eta_A(a),s)$, which is a functional relation from $S_A$ to $A$, as follows from the fact that $\eta_A$ is internally bijective.\\\\
Thus to prove the lemma it suffices to prove that also the right hand side of the inequality is a functional relation from $S_A$ to $A$. This is true since, by definition of comprehension, we have $\top_{S_A}(s) = \sigma_A(\lfloor\sigma_A\rfloor(s))$ and recalling that $\sigma_A$ is a shorthand for $\exists_{\{\delta_A\}}\top_A$, we have that $$\top_{S_A}(s)=\exists_{\{\delta_A\}}\top_A (\lfloor\sigma_A\rfloor(s)) = \exists a:A.\ \delta_{\mathbb{P}(A)}(\{\delta_A\}(a),\lfloor\sigma_A\rfloor(s))$$ Now apply point iii) of definition \ref{mizzega} on $a\in_A\lfloor\sigma_A\rfloor(s): S_A\arr \mathbb{P}(A)$.\end{proof}
\begin{lem}\label{trip3}
For every functional relation $F$ from $Y$ to $A$ there exists a unique morphism $h:Y\arr S_A$ with $$F(y,a) = \delta_{S_A}(h(y), \eta_A(a))$$
\end{lem}
\begin{proof}
Suppose $F$ in $P(Y\times A)$ is functional from $Y$ to $A$. Then $F^{op} = <\pi_2,\pi_1>^*F$ determines a morphism $\{F^{op}\}:Y\arr \mathbb{P}(A)$ sucht that
$$\exists a:A.\ a\in_A \{F^{op}\}(y) = \top_Y(y)$$$$a\in_A \{F^{op}\}(y) \wedge a'\in_A \{F^{op}\}(y) \le \delta_A(a,a')$$then, by point iii) of definition \ref{mizzega}, $\{F^{op}\}^*\sigma_A=\top_Y$ and therefore, by the universal property of comprehension, there exists $h:Y\arr S_A$ with $\lfloor\sigma_A\rfloor \circ h = \{F^{op}\}$. Thus, by lemma \ref{trip2} $$F(y,a)=a\in_A \{F^{op}\}(y) = a\in_A \lfloor\sigma_A\rfloor(h (y)) = \delta_{S_A}(\eta_A(a),h(y))$$To prove uniqueness of $h$, note that for every $g$ such that $\delta_{S_A}(\eta_A(a),g(y)) = F(y,a)$, it holds that  $$a\in_A \lfloor\sigma_A\rfloor(g (y))= \delta_{S_A}(\eta_A(a),g(y))=\delta_{S_A}(\eta_A(a),h(y))=a\in_A \lfloor\sigma_A\rfloor(h (y))$$Recall that the formula $a\in_A \lfloor\sigma_A\rfloor(g (y))$ corresponds to $(id_A\times\lfloor\sigma_A\rfloor\circ g)^*\in_A$, then from the previous equality we have $$(id_A\times\lfloor\sigma_A\rfloor\circ g)^*\in_A=(id_A\times\lfloor\sigma_A\rfloor\circ h)^*\in_A$$ and therefore $$\lfloor\sigma_A\rfloor\circ g=\{(id_A\times\lfloor\sigma_A\rfloor\circ g)^*\in_A\} = \{(id_A\times\lfloor\sigma_A\rfloor \circ h)^*\in_A\}= \lfloor\sigma_A\rfloor \circ h$$since comprehension morphisms are mono, $g=h$.
\end{proof}
We now want to prove that for an \ed\ doctrine $(\C,P)$ with singletons, the inclusion of the category \textbf{Shv}$(\C,P)$ in $\C$ has a left adjoint, whose unite is the family of morphisms of the form $\eta_A:A\arr S_A$. In order to prove this, knowing that, by lemma \ref{trip1}, every $\eta_A$ is internally bijective, it is enough to show that for every $A$ in $\C$ and every span of the form 
\[
\xymatrix{
Y&X\ar[l]_-{d}\ar[r]^-{q}&S_A
}
\]
where $d$ is internally bijective, there exists $h:Y\arr S_A$ with $h\circ d = q$. This infact shows that $S_A$ is a $P$-sheaf and moreover replacing $d$ with $\eta_X$ we have that every $q:X\arr S_A$ has a unique extention to $S_X$, proving that \textbf{Shv}$(\C,P)$ is reflective.
\begin{prop}\label{trip}
Every \ed\ doctrine with singletons admits sheafification.
\end{prop}
\begin{proof} Consider the span above and the following diagram
\[
\xymatrix{
X\ar[d]_-{d}\ar[r]^-{q}&S_A \ar[r]^-{\lfloor\sigma_A\rfloor}& \mathbb{P}(A)\\
Y\ar[rru]_-{\{\xi^{op}\}}&&
}
\] 
where $\xi(y,a) = \exists x:X.\ \delta_Y(y,d(x)) \wedge \delta_{S_ A}(q(x),\eta_A(a))$ and $\xi^{op} = <\pi_2,\pi_1>^*\xi$.\\\\ We need prove that there exists a unique $h:Y\arr S_A$ with $h\circ d = q$.\\\\ To prove the existence of $h$, note that, by Frobenius Reciprocity and the fact that $\eta_A$ and $d$ are internally surjective, $\top_Y(y) = \exists a:A.\ \xi(y,a)$. Again by Frobenius Reciprocity and by internal injectivity of $\eta_A$ and $d$ we have that $\xi(y,a) \wedge \xi(y,a) \le \delta_A(a,a')$. Therefore $\xi$ is functional from $Y$ to $A$ and by lemma \ref{trip3} there exists a unique $h:Y\arr S_A$, such that $\xi(y,a) = \delta_{S_A}(h(y),\eta_A(a))$.\\

Once we have the morphism $h: Y\arr S_A$, we need show that it makes the diagram commutes. Note, by using internal injectivity of $d$, that $\xi(d(x),a) = \delta_{S_A}(f(x),\eta_A(a))$ which means that $\delta_{S_A}(h(d(x)),\eta_A(a)) = \delta_{S_A}(f(x),\eta_A(a))$ and thus, by lemma \ref{trip3}, $h\circ d = f$.
\end{proof} 
Thus the diagram at the end of section \ref{sec7} becomes
\[
\xymatrix{
\textbf{Map}_c(\C,P)\ar@{^{(}->}[rr]\ar@/^/[rrd]^-{L}\ar[d]_-{U}&&\textbf{Map}(\C,P)\\
\textbf{Shv}(\C,P)\ar@{^{(}->}@/_/[rr]\ar@{}[rr]|{\ \ \ \bot}&&\C\ar[u]_-{\Gamma}\ar@/_/[ll]
}
\]
Moreover, the fact that for an \ed\ doctrine $(\C,P)$ with singletons the category \textbf{Shv}$(\C,P)$ is reflective allows to characterize $P$-sheaves as complete objects, as in the following corollary.
\begin{cor}\label{cca}
Let $(\C,P)$ be an \ed\ doctrine with singletons, then the functor $U$ is an equivalence.
\end{cor}
\begin{proof}$U$ is trivially full and faithful. To prove that it is essentially surjective, take $A$ in \textbf{Shv}$(\C,P)$ and a functional relation $F$ from $Y$ to $A$. Then $A$ is a complete object if there exists a morphism in $\C$ from $Y$ to $A$ whose internal graph is $F$. By lemma \ref{trip3} there exists a morphism $h:Y\arr S_A$ in $\C$ such that $F = (h\times\eta_A)^*\delta_{S_A}$. Since \textbf{Shv}$(\C,P)$ is reflective, $\eta_A$ is an isomorphism, then $\eta_A^{-1}\circ h:Y\arr A$ is the desired morphism.\end{proof}
\begin{cor}\label{ccaccu}
Let $(\C,P)$ be an \ed\ doctrine with singletons, then the inclusion of \textbf{Map}$_c(\C,P)$ into \textbf{Map}$(\C,P)$ is an equivalence.
\end{cor}
\begin{proof}Since the inclusion in full, we only have to show that it is also essentially surjective. Suppose $A$ is in \textbf{Map}$(\C,P)$, then by \ref{cca} we have that $S_A$ is in \textbf{Map}$_c(\C,P)$. Since $\eta_A:A\arr S_A$ is internally bijective it straightforward to show that the functional relation $\Gamma\eta_A$ from $A$ to $S_A$ is an isomorphism in \textbf{Map}$(\C,P)$.
\end{proof}
Thus the diagram above reduces to
\[
\xymatrix{
\textbf{Map}(\C,P)\simeq\textbf{Map}_c(\C,P)\simeq\textbf{Shv}(\C,P)\ar@/_/[d]\ar@{}[d]|{ \vdash}\\
\C\ar@/_/[u]
}
\]

\section{Applications}\label{sec7b}
We discuss in this section some relevant examples.\\\\
Let $\mathcal{E}$ be an elementary topos and $j$ a Lawvere-Thierney topology. Denote by $\mathcal{E}_j$ the sub-topos of $j$-sheaves, i.e. the full subcategory of $\mathcal{E}$ on those objects which are orthogonal to $j$-dense monomoprhisms \cite{TT}. Denote also by $(\mathcal{E}, \text{cl}_j)$ the doctrine of $j$-closed subobjects. The doctrine $(\mathcal{E},\text{cl}_j)$ is an \ed\ doctrine with singletons \cite{Jacobs}, therefore by \ref{trip} it admits sheafification. We show that $\mathcal{E}_j$ is equivalent to \textbf{Shv}$(\mathcal{E},\text{cl}_j)$.\\\\The class of internally bijective morphisms in $(\mathcal{E},\text{cl}_j)$ is the class of $j$-bidense arrows of $\mathcal{E}$. Thus the category \textbf{Shv}$(\mathcal{E},\text{cl}_j)$ is the full sub-category of $\mathcal{E}$ on those objects which are orthogonal to all $j$-bidense arrows. Since every $j$-dense mono is in particular $j$-bidense, \textbf{Shv}$(\mathcal{E},\text{cl}_j)$ is a full subcategory of $\mathcal{E}_j$. To prove that the inclusion is also essentially surjective, take $A$ in $\mathcal{E}_j$ and denote by $L$ the associated sheaf functor. $L$ is known to be left adjoint to the inclusion of $\mathcal{E}_j$ in $\mathcal{E}$ and denote by $l$ the unite of the adjunction. Take any morphism $q:X\arr A$ and any $j$-bidense morphism $d:X\arr Y$ then consider the following commutative diagram
\[
\xymatrix{
Y\ar[d]_-{l_Y}&X\ar[l]_-{d}\ar[d]_-{l_X}\ar[r]^-{q}&A\ar[d]_-{l_A}\\
LY&LX\ar[l]^-{Ld}\ar[r]_-{Lq}&LA
}
\]
Since $A$ is a $j$-sheaf, $l_A$ is iso. Also $Ld$ is iso since $d$ is $j$-bidense \cite{TT}, thus in particular is a $j$-dense mono. So there exists $h:LY\arr A$ with $h\circ Ld = l_A^{-1}\circ Lq$. Then we have a morphism $h \circ l_Y: Y\arr A$ making the diagram commute.\\To show its uniqueness, suppose two morphisms $g$ and $f$ are such that $g\circ d = q = f\circ d$, then, since $d$ is bidense, it is $j$-true that their graphs are equal, by the same argument as \ref{ccu}, and therefore it is $j$-true that $\top_Y = \delta_A(g(y),f(y))$. Since $A$ is a $j$-sheaf and therefore a separated object, its diagonal is closed and then $g = f$ \cite{TT}.\\\\
With the next class of example we want to show that given a tripos $(\C,P)$, the topos $\C[P]$, i.e. the topos obtained from $(\C,P)$ by the tripos to topos construction \cite{tripinret}, is a category of the form \textbf{Shv}$(\overline{\C}, \overline{P})$ for an appropriate doctrine $(\overline{\C}, \overline{P})$ built out of $(\C,P)$. For the definition of tripos, the reader is referred to \cite{tripinret}.\\\\
Suppose $(\C,P)$ is a tripos and consider the pair $(\overline{\C}, \overline{P})$ obtained from $(\C,P)$ by freely adding comprehensions, extensional equality and quotients: details can be found in  \cite{RM2}, nevertheless we give an explicit description of $(\overline{\C}, \overline{P})$.\\\\
Objects of $\overline{\C}$ are pairs $(A,\rho)$ where $A$ is an object of $\C$ and $\rho$ is a partial equivalence relation over $A$, i.e. an formula of $P(A\times A)$ such that $\rho(x,y) = \rho(y,x)$ and $\rho(x,y)\wedge\rho(y,z)\le\rho(x,z)$.\\\\
A morphism $[f]:(A,\rho)\arr(B,\sigma)$ is an equivalence class of morphisms $f:A\arr B$ of $\C$ such that $\rho\le(f\times f)^*\sigma$, with respect to the following equivalence relation $$f\sim g\ \ \ \ \ \ \text{if and only if}\ \ \ \ \ \ \top_A\le<f,g>^*\sigma$$
For an object $(A,\rho)$ in $\overline{\C}$ we have that $$\overline{P}(A,\rho) = \{\phi\ \epsilon\ P(A)\mid\phi(x)\le\rho(x,x)\ \ \text{and}\ \ \phi(x)\wedge\rho(x,y)\le\phi(y)\}$$
For a morphism $[f]:(A,\rho)\arr(B,\sigma)$ we have that $\overline{P}[f]$ is give by the assigment $$\phi\mapsto f^*\phi\wedge\Delta_A^*\rho$$ for $\phi$ in $\overline{P}(A,\rho)$. The doctrine $(\overline{\C}, \overline{P})$ is a tripos with power objects \cite{FP2}.\\\\
It is straightforward to see that the topos $\C[P]$ obtained from $(\C,P)$ via the tripos to topos construction is \textbf{Map}$(\overline{\C},\overline{P})$, then by \ref{ccaccu} we have $\C[P]\simeq\textbf{Shv}(\overline{\C},\overline{P})$.

\end{document}